\documentclass[USenglish]{article}	
\usepackage[utf8]{inputenc}				%(only for the pdftex engine)
\usepackage[big,online]{dgruyter}	%values: small,big | online,print,work
\usepackage{lmodern} 
\usepackage{microtype}
\usepackage[numbers,square,sort&compress]{natbib}

\usepackage{amsmath}
\usepackage{amsfonts}
\usepackage{amssymb}
\usepackage{amsthm}
\usepackage{cuted}
\usepackage{xcolor}
\usepackage{enumitem}
\usepackage{url}

\newtheorem{thm}{Theorem}
\newtheorem{cor}{Corollary}[section]
\newtheorem{lem}{Lemma}[section]
\newtheorem{prop}{Proposition}[section]

\newtheorem{rem}{Remark}
\newcommand{\comment}[1]{}

\newcommand{\Real}{\mathbb R}

\newcommand{\bbP}{\mathbb P}
\newcommand{\bbQ}{\mathbb Q}
\newcommand{\Lp}{L^\bbp}
\newcommand{\Lq}{L^\bbq}

\newcommand{\bbQP}{{\bbQ\times\bbP}}

\newcommand{\bbt}{{2}}
\newcommand{\Lt}{{L^\bbt}}

\newcommand{\Hamil}{\mathbf {H}}

\newcommand{\bbp}{\mathfrak{p}}
\newcommand{\bbq}{\mathfrak{q}}

\newcommand{\bbg}{\mathfrak{g}}
\newcommand{\bbf}{\mathfrak{f}}

\newcommand{\cT}{\mathcal{T}}
\newcommand{\cS}{\mathcal{S}}

\newcommand{\hs}{{h^*}}
\newcommand{\cTs}{{\cT^\dag}}

\newcommand\scpr[2]{\langle #1, #2 \rangle}

\theoremstyle{dgthm}
\newtheorem{theorem}{Theorem}

\theoremstyle{dgdef}
\newtheorem{definition}{Definition}

\begin{document}

%%%--------------------------------------------%%%
	\articletype{Research Article}
	\received{Month	DD, YYYY}
	\revised{Month	DD, YYYY}
  \accepted{Month	DD, YYYY}
  \journalname{De~Gruyter~Journal}
  \journalyear{YYYY}
  \journalvolume{XX}
  \journalissue{X}
  \startpage{1}
  \aop
  \DOI{10.1515/sample-YYYY-XXXX}
%%%--------------------------------------------%%%

\title{On $L^q$ Convergence of the Hamiltonian Monte Carlo}
\runningtitle{$L^q$ Convergence of HMC}
%\subtitle{Insert subtitle if needed}

\author[1]{Soumyadip Ghosh}
%\ use * to mark the author as the corresponding author
\author[1]{Yingdong Lu}
\author*[1]{Tomasz Nowicki} 
\runningauthor{S. Ghosh et al.}
\affil[1]{\protect\raggedright 
IBM T.J. Watson Research Center, Yorktown Heights, NY, U.S.A., e-mail: \{ghoshs,yingdong,tnowcki\}@us.ibm.com}
%\affil[2]{\protect\raggedright 
%Institution, Department, City, Country of second author and third author, e-mail: author\_two@xx.yz, author\_three@xx.yz}
	
%\communicated{...}
%\dedication{...}
	
\abstract{We establish $L_\bbq$ convergence for Hamiltonian Monte Carlo algorithms. More specifically, under mild conditions for the associated Hamiltonian motion, we show that the outputs of the algorithms converge (strongly for $2\le \bbq<\infty$ and weakly for $1<\bbq<2$) to the desired target distribution.
}

\keywords{Convergence, $\Lq$ spaces, Hamiltonian Monte Carlo}

\maketitle

\section{Introduction}

\noindent
Hamiltonian Monte Carlo or Hybrid Monte Carlo (HMC) algorithm is a method to obtain random samples from a
(target) probability distribution 
%$\bbf/\int_\bbQ\bbf$ 
on a space $\bbQ$ whose density is known only up to a factor. The target distribution is expressed as $\bbf/\int_\bbQ\bbf$ with function value $\bbf$ known, but the normalizing constant $\int_\bbQ\bbf$ is not known or at least is prohibitively expensive to calculate.
It is an algorithm of the Metropolis-Hastings type known for a while~\cite{DuaneEtAl87} for estimating integrals.
Convergence in various probability senses can be found in the literature, see, e.g. ~\cite{livingstone2019,ADHMC}.
In~\cite{GHOSH2022107811}, convergence of the densities in $\Lt$ Hilbert space was obtained. Meanwhile, utilizing different norms, in addition to providing great flexibility in terms of model selections in the applications of machine learning algorithms, can have significant impact in the behavior of machine learning algorithms, see, e.g.~\cite{hastie2015statistical}.
In this paper, we address the problem of convergence of the algorithm for densities in $\Lq$ spaces.
\comment{Our goal is to provide a clear and understandable reason why HMC algorithm converges to the right limit for densities in the spaces $\Lq(\bbQ)$, for $1<\bbq<\infty$. These convergence results not only paint a more complete picture of the HMC method, thus deepen its understanding, but also provide great flexibility in terms of model selections in the applications of machine learning algorithms.} 
%We refer to our papers~\cite{3,4,5} for three approaches to investigation of HMC, analytic (in $L^2$), probabilistic and  algorithmic.
%Here we concentrate on an extension of the analytic one, the convergence in $\Lq$.

In each HMC iterative step, the input sample point in space $\bbQ$ is lifted to the product space $\bbQP$ by adding a sample from an auxiliary distribution of choice~$\bbg$, where $\bbg>0$ on $\bbP$ and $\int_\bbP\bbg=1$, then the pair is transformed via the Hamiltonian motion generated by the Hamiltonian energy $\Hamil(q,p)=-\log(\bbf(q)\cdot\bbg(p))$ and projected back to $\bbQ$ as the input of the next iteration. Details can be found in  ~\cite{GHOSH2022107811}.

\medskip

Viewing as transformation of distributions,  one can present HMC as follows: Given some initial distribution $h(q)$ on $\bbQ$ one produces a joint distribution $(h\cdot\bbg)(q,p)=h(q)\cdot \bbg(p)$ on the phase space $\bbQP$ (joining them as independent) then the points are moved by the motion, $(q,p)\mapsto (Q,P)=H(q,p)$ producing another distribution $(h\cdot\bbg)\circ H(q,p)=h(Q)\cdot(P)$ in $\bbQP$, and finally one projects it to a distribution on $\bbQ$
by calculating the marginal $\int_\bbP (h\cdot\bbg)\circ H(q,p)\,dp$, which is a result of the action of the algorithm in one step. In short
\begin{equation}
  \label{eqdef cT}
  \cT(h)(q)=\int_\bbP (h\cdot \bbg)\circ H (q,p)\,dp\,,
\end{equation}
 and from a rather complicated algorithm we receive a relatively simple, linear operator on some space of integrable functions. The convergence of the algorithm thus corresponds to the convergence of the sequences of iterates of $\cT$.
\comment{
A few remarks.
\begin{itemize}
\item It is clear that the Hamiltonian motion $H$ does not depend on the constant factor in front of $\bbf$.
\item The motion $H$ in practical implementation is calculated by the \emph{leap-frog algorithm} which displays the needed invariance properties. We shall not deal with it in this paper.
\item An example of the situation where the target distribution is known up to a normalizing constant is the Bayesian update. In order to establish the distribution of (random)
    parameters $\theta$ influencing the outcome $D$ of the observations, when we know all the conditional probabilities $P(D|\theta)$ one uses the knowledge of the outcome of an experiment $D$ to improve the estimate:
    Given the estimate of the target distribution $\pi(\theta)$ \emph{prior} to the experiment we calculate $\hat{\pi}(\theta)=
    P_\pi(\theta|D)=P_\pi(\theta, D)/P_\pi(D)=P(D|\theta)\cdot\pi(\theta)/\sum_{\theta'} P(D|\theta')\cdot\pi(\theta')$ and take $\hat{\pi}$ as a new estimate \emph{posterior} to the experiment. However the sum (integral) in the denominator may be not that easy to calculate. This yields to $\hat{\pi}(\theta)\sim P(D|\theta)\pi(\theta)$ without the normalizing factor.
\end{itemize}
}

\section{Results}
Assume that the motion $H:\bbQP\to\bbQP$, $H(q,p)=(Q,P)$ satisfies the following \textbf{invariance properties}:
\begin{itemize}
  \item $(\bbf\cdot\bbg)\circ H=\bbf\cdot\bbg$;
  \item $\iint_\bbQP A\circ H =\iint_\bbQP A$ for any integrable $A$;
  \item $Q(q,\bbP)=\bbQ$ for (almost) every $q$;
  \item $\bbQ$ is the support of $\bbf$.
\end{itemize}
 For $\cT h=\int_\bbP (h\cdot\bbg)\circ H$ let $\cT^{n+1}=\cT^n\circ\cT$.
 The \emph{adjoint} operator $\cTs$ is given by the same formula~\eqref{eqdef cT} with $H^{-1}$ in place of $H$ and is described below in Section~\ref{sec:adjoint}, a \emph{self-adjoint} operator satisfies $\cTs=\cT$.
 
 Let $\Lq$ denote the space of  functions $h:\bbQ\to\Real$ such that $||h||_\bbq^\bbq=\int_\bbQ |h|^\bbq/\bbf^{\bbq-1}<\infty$ and the support of $h$ is included in the support of $\bbf$, which we may assume to be  equal to $\bbQ$.
  \begin{thm}
   Assume the invariance properties and assume that the operator $\cT$ is self-adjoint. For every $h\in\Lq(\bbQ)$, $1<\bbq<\infty$ the sequence $\cT^n h$ converges weakly in $\Lq$ to
   $\bbf \cdot \int h/\int\bbf$ and for $2\le \bbq<\infty$ it converges also strongly. %Moreover, when $h$ is a density, so is the limit.
 \end{thm}
 \begin{proof}

%The proof is not very hard. 
We observe that $\cT$ is in fact (Lemma~\ref{lem:cT prop}.\eqref{lem:cT prop alt}) an averaging map, thus by the convexity of $x\mapsto x^\bbq$, $\bbq> 1$, the norm of $h$ decreases (Lemma~\ref{lem:cT prop}.\eqref{lem:cT prop norm}) under $\cT$, sharply unless (by coverage assumption) $h=\alpha\bbf$. The spaces $\Lq$, $1<\bbq<\infty$ are reflexive, hence bounded sequences have weak accumulation points. Using self-adjointness and the convexity of $x\mapsto x^{\bbq-1}$, for $\bbq\ge2$ (Lemma~\ref{lem: cT h s and cT hs}) we prove that each accumulation point must be of form $\alpha\bbf$, proving (Corollary~\ref{cor: weak conv}) weak convergence for $\bbq\ge2$. Meanwhile the proof of the convergence of the norms provides (Corollary~\ref{cor: strong conv}) strong convergence. Weak convergence for $1<\bbq<2$ follows (Proposition~\ref{prop:weak conv}) from special properties of the fixed point $\bbf$.
\end{proof}
\begin{rem}
\
\begin{itemize}
\item The Hamiltonian motion satisfies two first integral invariance assumptions, as both the Hamiltonian  and the Lebesgue measure are invariant under such a motion.
\item
    The covering  property $Q(q,\bbP)=\bbQ$ can be weakened to a statement of an eventual coverage, not necessarily in one step. Some type of irreducibility must be assumed to avoid complete disjoint domains of the motion and hence an obvious non existence of a (unique) limit.
%\item The support condition takes care of some initialization problems with the division by 0. This can be formally avoided by working in the space of likelihoods, see below.
\end{itemize}
\end{rem}

\section{The properties of $\cT$ in $\Lq$}

We shall be working in the reflexive spaces $\Lq(\bbQ)$ and its dual $\Lp(\bbQ)$,
where $\bbq,\bbp>1$ are conjugated
real numbers  $\bbq+\bbp=\bbq\cdot\bbp$.
%We remark that $1/\bbp=1-1/\bbq=(\bbq-1)/\bbq$ and $\bbq=(\bbq-1)\bbp$.
In such spaces for $h\in\Lq$ we have standard:
\begin{eqnarray*}
  \text{norm}&& ||h||_\bbq^\bbq=\int_\bbQ \left|\frac{h}{\bbf}\right|^\bbq \bbf,
  \\
\text{bilinear form }&&\scpr{\cdot}{\cdot}:\Lq\times\Lp\to\Real:\scpr{a}{b}=\int_\bbQ\frac{a\cdot b}{\bbf},
 \\
  \text{and conjugacy }&&*:\Lq\to\Lp, \hs=h\cdot\left(\frac{|h|}{\bbf}\right)^{\bbq-2}
\,.
\end{eqnarray*}
We shall assume, unless stated otherwise, that $h\ge 0$.
\comment{
Call $h/\bbf$ a likelihood (up to an irrelevant  normalizing constant $\int\bbf$) of $h$ with respect to~$\bbf$.
The space  $\widetilde{\Lq}=\{\tilde{h}:\int_\bbQ |\tilde{h}|^\bbq \bbf<\infty\}$ of likelihoods $\tilde{h}=h/\bbf$
with the norm $||\tilde{h}||_\bbq^\bbq=\int_\bbQ |\tilde{h}|^\bbq\cdot\bbf$ is isometric to $\Lq$.}
\begin{lem}\label{lem:prop of nbc}
\begin{eqnarray}
\label{lem:cT prop sc}
a\in \Lq, b\in\Lp&\Rightarrow&
\scpr{a}{b}\le ||a||_\bbq\cdot||b||_\bbp,\\
\label{lem:cT prop hs}
\hs\in \Lp;\quad ||h||_\bbq^\bbq&=&\scpr{h}{\hs}=||\hs||_\bbp^\bbp;\quad (\hs)^*=h,\\
\label{lem:cT prop f}
\bbf\in\Lq;\quad ||\bbf||_\bbq^\bbq&=&\int_\bbQ \bbf;\quad \bbf^*=\bbf,\\
\label{lem:cT prop int f}
\scpr{h}{\bbf}=\int_\bbQ h,\quad h\in\Lq;&&\scpr{\bbf}{\hs}=\int_\bbP \hs,\quad \hs\in\Lp\,.
\end{eqnarray}
\end{lem}
\begin{proof}
 \eqref{lem:cT prop sc} is the H\"older inequality for likelihoods $a/\bbf$ and $b/\bbf$. 
 %in the spaces $\tilde{\Lq}$ and $\tilde{\Lp}$ .
 \\
\eqref{lem:cT prop hs} follows a straightforward calculation using $\bbq(\bbp-1)=\bbp$. We note that $h\cdot \hs\ge 0$.\\
\eqref{lem:cT prop f} and \eqref{lem:cT prop int f} follow directly from the definitions.
\end{proof}

\begin{lem}[Properties of $\cT$]\label{lem:cT prop}
For $0\le h\in\Lq$:
\begin{eqnarray}
\label{lem:cT prop alt}
\cT h &=&  \bbf\cdot\int_\bbP \frac{h}{\bbf}\circ H \cdot\bbg,
\\
\label{lem:cT prop int}
\int_\bbQ\cT h &=&  \int_\bbQ h,\\
\label{lem:cT prop norm}
||\cT h||_\bbq &\le& ||h||_\bbq\,.
\end{eqnarray}
The equality in \eqref{lem:cT prop norm} occurs iff $h=\alpha\cdot\bbf$ ($\bbf$ a.e.), where $\alpha=\alpha(h)=\int h/\int \bbf$.
\end{lem}
\begin{proof}
\eqref{lem:cT prop alt}:
Using the invariance properties we have
$\int_\bbP \frac{h}{\bbf}\circ H\cdot (\bbf\cdot\bbg)\circ H=
 \int_\bbP \frac{h}{\bbf}\circ H \cdot (\bbf\cdot\bbg)
$
and $\bbf$ does not depend on $p\in\bbP$.\\
\eqref{lem:cT prop int}:
$\int_{\bbQ}\int_{\bbP} (h\cdot\bbg)\circ H =\iint_{\bbQP}(h\cdot \bbg)=
  \left(\int_{\bbQ} h\right)\left(\int_{\bbP}\bbg\right)$.\\
\eqref{lem:cT prop norm}:
$
||\cT h||_\bbq^\bbq=\int_\bbQ \left|\int_\bbP\frac{h}{\bbf}\circ H\cdot\bbg\right|^{\bbq}\bbf
\le
\int_\bbQ \int_\bbP\left|\frac{h}{\bbf}\circ H\right|^{\bbq}\bbg\,\bbf
 =
\iint_\bbQP\left|\frac{h}{\bbf}\right|^\bbq\circ H\cdot(\bbg\cdot\bbf)\circ H
 =
\iint_\bbQP\left|\frac{h}{\bbf}\right|^\bbq\bbg\cdot\bbf
$ the last one being equal to
$\left(\int_\bbQ\left|\frac{h}{\bbf}\right|^\bbq\bbf\right)\cdot\left(\int_\bbP\bbg\right)=
||h||_\bbq^\bbq$, as $\bbg\ge0$ and $\int_\bbP\bbg=1$. Given $q$ the equality occurs only if $(h/\bbf)(H(q,p))$ is a constant for ($\bbg$-)almost all $p$, but by coverage assumption it means that $h/\bbf$ is a constant on ($\bbf$-)almost all $\bbQ$. The constant follows from~\eqref{lem:cT prop int}.\\
\end{proof}
Remark that inequality~\eqref{lem:cT prop norm} is valid also in the boundary cases of $\bbq=1$ and $\bbq=\infty$ (the sup norm). However in case of $\bbq=1$ the equality occurs for all positive functions.

The operator $\cT$ is an averaging operator of the (transported) likelihood $h/\bbf$ with respect to the probability $\bbg$. The scalar functional is monotone:
$0\le a\le b,\quad 0\le c\le d$ implies $\scpr{a}{c}\le \scpr{b}{d}$ and
$\cT$ is positive, in particular if $a\le b$ then $\cT a\le \cT b$.
The function  $\bbf$ provides the eigendirection of fixed points and by \eqref{lem:cT prop norm}
$\cT$ has its spectrum in the unit disk.
The eigenvalue 1 is a unique eigenvalue on the unit circle and it has multiplicity 1.
For any $h\in \Lq$ one has the unique decomposition $h=\alpha\bbf+(h-\alpha\bbf)$ where $\alpha\bbf$ is a direction of the fixed points and
$h-\alpha\bbf\in N:=\{a\in\Lq:\int a =0\}$ lies in an invariant subspace.
It is not \emph{a~priori} clear under what conditions $1$ is isolated in the spectrum, in other words whether the contraction $||\cT h||<||h||$ is uniform on $N$, which would imply $\cT^n N\to \{0\}$ (point-wise) with exponential speed.
\begin{lem}\label{lem: cT h s and cT hs}
  For $2\le \bbq<\infty$ and $\Lq\ni h\ge 0$
  \begin{equation*}\label{eqn: cT h s and cT hs}
    (\cT^n h)^*\le \cT^n(\hs)
  \end{equation*}
  For $1<\bbq\le 2$ the opposite inequality holds.
   The equality happens when $\bbq=2$ or when $h$ is aligned with $\bbf$.
\end{lem}
\begin{proof}
The comparison acts in $\Lp$. It is enough to prove for $n=1$, as the general case follows by induction and positivity of $\cT$.
Proof for $n=1$ follows from the  convexity of $x\mapsto x^{\bbq-1}$, positivity of the linear operator
$\cT$ and its averaging property.
Again the inequality is sharp unless $h=\alpha\bbf$.
\end{proof}

\section{The adjoint operator $\cTs$}\label{sec:adjoint}
As $H$ is invertible the inverse map $H^{-1}$ is well defined and it enjoys the same invariance properties as $H$. Define
\[
{\cTs}h=\int_\bbP (h\cdot\bbg)\circ H^{-1}.
\]
It enjoys the same properties as $\cT$ enumerated in Lemmata above. 
%It turns out that
%the operator ${\cTs}$ defined by $H^{-1}$:
It is conjugated to $\cT$ with respect to the duality functional $\scpr{\cdot}{\cdot}$, namely
\begin{lem}
  For $h\in\Lq$ and $k\in\Lp$:
  \begin{equation}\label{lem: cTs}
    \scpr{\cT h}{k}=\scpr{h}{{\cTs} k}
  \end{equation}
\end{lem}
\begin{proof}
  Using \eqref{lem:cT prop alt} and invariance
$\scpr{\cT h}{k}=\int_\bbQ\left(\int_\bbP \frac{h}{\bbf}\circ H\cdot\bbg\right)\cdot k=
\iint_{\bbQP}\frac{h}{\bbf}\cdot(\bbg\cdot k)\circ H^{-1}=
\iint_{\bbQP}\frac{h}{\bbf}\cdot(\frac{k}{\bbf})\circ H^{-1}\cdot(\bbg\cdot\bbf)\circ H^{-1}=
\iint_{\bbQP}\frac{h}{\bbf}\cdot(\frac{k}{\bbf}\circ H^{-1})\cdot(\bbg\cdot\bbf)=
\int_{\bbQ}{h}\cdot\left(\int_\bbP\frac{k}{\bbf}\circ H^{-1}\cdot\bbg\right)=\scpr{h}{\cTs k}\,.
$
\end{proof}

%\subsection*{Case of self-adjoint operator,  when $\cT={\cTs}$}

The following Lemma provides a sufficient condition for $\cT$ to be self-adjoint. 

Let $\sigma$ be a measure preserving involution $\sigma:\bbP\to\bbP$, $\sigma\circ\sigma={\rm id}$.
We can extend it to $\sigma:\bbQP\to\bbQP$ by $\sigma(q,p)=(q,\sigma(p)$).
Assume that $\bbg$ is invariant with respect to $\sigma$: $\bbg\circ\sigma=\bbg$.
\begin{lem}\label{lem:sigma}
If $\sigma\circ H^{-1}\circ\sigma=H$ and $\bbg$ is invariant with respect to $\sigma$ then ${\cTs}=\cT$.
\end{lem}

\begin{proof}
Measure invariance means that $\int_\bbP a\circ\sigma=\int_\bbP a$. Let $(Q,P)=H^{-1}(q,p)$ then
  $\sigma\circ H^{-1}(q,p)=\sigma(Q,P)=(Q,\sigma (P))$ and
  $\cT h=\int_\bbP (h\cdot \bbg)\circ H=\int_\bbP (h\cdot \bbg)\circ \sigma\circ H^{-1}\circ \sigma=
  \int_\bbP (h\cdot \bbg)\circ \sigma\circ H^{-1}=\int_\bbP(h\cdot\bbg)\sigma(Q,P)=
  \int_\bbP (h(Q)\cdot \bbg(\sigma(P))=\int_\bbP h(Q)\bbg(P)=\int_\bbP h\circ H^{-1}\cdot\bbg\circ H^{-1}=\cTs h$.
\end{proof}
As an example take $\bbQ=\bbP=\Real$,
$\sigma$ to be the symmetry (reflection) of the space $\bbP$ with respect to 
0, $\sigma(p)=-p$. An even density $\bbg(p)=\bbg(-p)$ is invariant with respect to $\sigma$.
The Hamiltonian motions $H$ and
$H^{-1}$ satisfy the condition of the Lemma.
\comment{
In what follows we shall assume that
%\begin{equation}\label{eqn:self adj}
 $ \cT=\cTs$
%\end{equation}
If it is not the case we can use in the algorithm the operator $\cS={\cTs}\circ\cT$, as ${\cS^\dag}=\cS$.
}
\section{Limits of the sequences $\cT^n$ for a self-adjoint operator $\cT$}
In this section we assume that
%\begin{equation}\label{eqn:self adj}
 $ \cT=\cTs$.
%\end{equation}
If it is not the case we can use in the algorithm the operator $\cS={\cTs}\circ\cT$, as ${\cS^\dag}=\cS$.

From  $||\cT h||_\bbq<||h||_\bbq$ by induction we obtain $||\cT^nh||_\bbq<||h||_\bbq$ unless $h=\alpha\bbf$, when equality holds.
For $h\in\Lq$ define
\begin{equation*}\label{eqndef: V}
V_\bbq(h)=\inf ||T^n h||_\bbq^\bbq =\lim ||T^n h||_\bbq^\bbq\,.
\end{equation*}
We see that $V_\bbq(h)=V_\bbq(\cT^n(h))$.
As we are interested in the limit of the sequence $\cT^n h$, for a given $h$ we can assume
that for an arbitrary $\epsilon>0$ we have $||h||_\bbq^\bbq<V_\bbq+\epsilon$, taking a high iterate $\cT^M h$ instead of $h$ if needed.

By a corollary to Alaoglu Theorem bounded sets in reflexive $\Lq$ are weakly (the same as weakly*) compact.
Let $h_\infty$ denote any  weak accumulation point (limit of a subsequences) of $\cT^n h$, say $\cT^{m_n}h\rightharpoonup h_\infty$.
We may assume that the subsequence $(m_n)$ has infinite number of even numbers,
otherwise take $\cT h$ in place of $h$. Then $h_\infty$ is the weak limit of the subsequence indexed by these even numbers.
We shall simplify the notation and use the indices $2m$ for this subsequence.
With this notation we have by the definition of weak convergence that $\scpr{\cT^{2m} h}{b}\to\scpr{h_\infty}{b}$ for every $b\in \Lp$.
\begin{prop}
Assume $\cT=\cTs$. Let $h_\infty$ be a weak limit of a subsequence $\cT^{m_n}(h_0)$, $0\le h_0\in\Lq$, $\bbq\ge 2$.
Then $||h_\infty||_\bbq^\bbq=V_\bbq(h_0)$.
\end{prop}
\begin{proof}
Denote $V=V_\bbq(h_0)$
  Let $\epsilon>0$ and $M$ be one of the indices in the weak converging subsequence large enough so that $h=\cT^M(h_0)$ has the norm $V\le ||h||_\bbq^\bbq\le
  V+\epsilon$, which is possible by definition as $V_\bbq(h)=V$. Then the sequence
$\cT^{m_n-M}(h) \rightharpoonup h_\infty$ weakly  and taking a subsequence and $\cT h $ instead of $h$ if needed, we can assume that the sequence $m_n-M$ consists of infinitely many positive even integers $2m$.
By Lemma~\ref{lem: cT h s and cT hs} and $\bbq/\bbp=\bbq-1$ we have
$
V\le ||\cT^m(h)||_\bbq^\bbq=\scpr{\cT^m( h)}{(\cT^m h)^*}\le
\scpr{\cT^m (h)}{\cT^m (\hs)}= \scpr{\cT^{2m} (h)}{\hs}\to
\scpr{h_\infty}{\hs}\le ||h_\infty||_\bbq\cdot||\hs||_\bbp=
 ||h_\infty||_\bbq\cdot(||h||_\bbq^\bbq)^{1/\bbp}
 =
 ||h_\infty||_\bbq\cdot(||h||_\bbq)^{\bbq-1}\le
 ||h_\infty||_\bbq\cdot(V+\epsilon)^{1-1/\bbq}
 $ thus by the arbitrary choice of $\epsilon>0$ we have $||h_\infty||_\bbq^\bbq\ge V$.
 The opposite direction is standard, we use \eqref{lem:cT prop sc} and \eqref{lem:cT prop hs}:
 $||h_\infty||_\bbq^\bbq=\scpr{h_\infty}{(h_\infty)^*}\leftarrow \scpr{\cT^{2m}(h)}{(h_\infty)^*}
 \le ||\cT^{2m}(h)||_\bbq\cdot||(h_\infty)^*||_\bbp\le ||h||_\bbq\cdot||h_\infty||_\bbq^{\bbq/\bbp}
 \le (V+\epsilon)^{1/\bbq}||h_\infty||_\bbq^{\bbq-1}$.
 \end{proof}
\begin{cor}\label{cor: weak conv}
  For $\bbq\ge 2$ every weak convergent subsequence of $\cT^n(h_0)$, $h_0\in\Lq$
  has a limit of norm $V_\bbq(h_0)$.
  In consequence $\cT^n(h_0)\rightharpoonup \alpha\bbf$.
\end{cor}
\begin{proof}
 If $\cT^{m_n}(h_0)\rightharpoonup h_\infty$ then $\cT^{m_n+1}(h_0)\rightharpoonup \cT(h_\infty)$ (use the operator $\cTs$). As they have the same norm, by Lemma~\ref{lem:cT prop}.\eqref{lem:cT prop norm} they are equal $h_\infty=\alpha\bbf$. Therefore every weakly convergent subsequence converges to the same limit, and as every subsequence has a weakly convergent subsequence the whole sequence converges.
\end{proof}
\begin{cor}\label{cor: strong conv}
For $\bbq\ge 2$ for each $h_0$ the sequence $\cT^n (h_0)$ converges strongly to $\alpha \bbf$, where
$\alpha=\int_\bbQ h_0/\int_\bbQ \bbf$.
\end{cor}
\begin{proof}
  Due to the strong convexity of the ball in $\Lq$ weak convergent sequence with the convergence of the norms to the norm of the limit convergences strongly.
\end{proof}
\begin{prop}\label{prop:weak conv}
  For any $1<\bbq<\infty$ and $h\in\Lq$ the sequence $\cT^n (h)$ converges weakly to $\alpha(h) \bbf$, where $\alpha(h)=\int_\bbQ h/\int_\bbQ\bbf$.
\end{prop}
\begin{proof}
The case $\bbq\ge 2$ follows from Corollary~\ref{cor: weak conv}.
Let $1<\bbq\le 2$, then $\bbp\ge 2$ and for any $a\in\Lp$ we have $\cT^m a\rightharpoonup \alpha(a)\bbf$, where $\alpha(a)=\int a/\int \bbf$. Let $h\in\Lq$ with $q\le 2$ and $a\in\Lp$. We have
$\scpr{\cT^n h}{a}=\scpr{h}{\cT^n a}\to\scpr{h}{\alpha(a)\bbf}=\alpha(a)\int h=(\int h\cdot \int a)/\int\bbf=\scpr{\alpha(h)\bbf}{a}$. Which means $\cT^n h\rightharpoonup \alpha(h)\bbf$.
\end{proof}
%%%%%%---------------------------------------
%\begin{thebibliography}{1}
%  \bibitem{1}
%   Duane, Simon; Kennedy, Anthony D.; Pendleton, Brian J.; Roweth, Duncan (3 September 1987). \emph{Hybrid Monte Carlo}. Physics Letters B. 195 (2): 216–222.
%  \bibitem{2}
%  Livingstone, Samuel; Betancourt, Michael; Byrne, Simon; Girolami, Mark. emph{On the geometric ergodicity of Hamiltonian Monte Carlo}. Bernoulli 25 (2019), no. 4A, 3109--3138.
%  \bibitem{3}
%  S.~Ghosh, Y.~Lu, T.~Nowicki \emph{HMC, an  Algorithms in Data Mining, the Functional Analysis approach.} in preparation.
%  \bibitem{4}
%  S.~Ghosh, Y.~Lu, T.~Nowicki \emph{HMC, an  Algorithms in Data Mining, the Probabilistic approach.} in preparation.
%  \bibitem{5}
%  S.~Ghosh, Y.~Lu, T.~Nowicki \emph{HMC, an  Algorithms in Data Mining, the Algorithmic approach.} in preparation.
%\end{thebibliography}
% ----------------------------------------------------------------
%\bibliographystyle{amsplain}
\bibliography{hmc}
\end{document}